\documentclass[reqno, 12pt]{amsart}

\usepackage{amssymb}
\usepackage{color}
\usepackage[dvips]{hyperref}

\newtheorem{theorem}{Theorem}[section]
\newtheorem{corollary}[theorem]{Corollary}
\newtheorem{lemma}[theorem]{Lemma}

\theoremstyle{definition}

\theoremstyle{definition}

\theoremstyle{definition}
\newtheorem{assumption}[theorem]{Assumption}

\makeatletter
\def\dashint{\operatorname%
{\,\,\text{\bf--}\kern-.98em\DOTSI\intop\ilimits@\!\!}}
\makeatother

\newcommand{\WO}[2]{\overset{\scriptscriptstyle0}{W}\,\!^{#1}_{#2}}

\def\sfa{{\sf a}}
\def\sfb{{\sf b}}
\def\sfc{{\sf c}}

\def\bR{\mathbb{R}}

\def\fL{\mathfrak{L}}

\def\fq{\mathfrak{q}}

\def\cD{\mathcal{D}}

\def\cO{\mathcal{O}}

\def\cL{\mathcal{L}}

\title[Parabolic equations in $L_p$-spaces with mixed norms]{
Parabolic equations with measurable coefficients in $L_p$-spaces with mixed norms}
\author{Doyoon Kim}

\address{Department of Mathematics and Statistics,
University of Ottawa,
585 King Edward Avenue,
Ottawa, Ontario K1N 6N5 Canada} 
\email{kdoyoon@uottawa.ca}



\begin{document}

\begin{abstract}
The unique solvability of parabolic equations in Sobolev spaces with mixed norms
is presented.
The second order coefficients (except $a^{11}$) are 
assumed to be only measurable in time and one spatial variable,
and VMO in the other spatial variables.
The coefficient $a^{11}$ is measurable in one spatial variable
and VMO in the other variables.
\end{abstract}

\maketitle

\section{Introduction}

In this paper we study parabolic equations of non-divergence type in $L_p$-spaces with mixed norms.
Specifically, we consider equations of the form
\begin{equation} \label{para_main_eq}
u_t + a^{ij}(t,x) u_{x^i x^j} + b^{i}(t,x) u_{x^i} + c(t,x) u = f
\end{equation}
in $L_{q,p}((S,T) \times \bR^d)$, $-\infty \le S < T \le \infty$,
where 
$$
\| u \|_{L_{q,p}((S,T) \times \bR^d)}
= \left( \int_S^T \left( \int_{\bR^d} |u(t,x)|^p \, dx \right)^{q/p} \, dt \right)^{1/q}.
$$

We assume that the coefficients $a^{ij}$, $i \ne 1$ or $j \ne 1$, are only measurable in $(t,x^1) \in \bR^2$ and VMO in $x' \in \bR^{d-1}$,
and the coefficient $a^{11}$ is measurable in $x^1 \in \bR$
and VMO in $(t,x') \in \bR^d$.
This means, for example, that no regularity assumptions on $a^{ij}$
are needed if $a^{ij}$, $i \ne 1$ or $j \ne 1$, are functions of only $(t,x^1)$ and $a^{11}$ is a function of only $x^1 \in \bR$.
As usual, the coefficients $b^{i}(t,x)$ and $c(t,x)$ are assumed to be only measurable and bounded.

Under these assumptions as well as the uniform ellipticity condition on $a^{ij}$,
we prove that, for $q > p \ge 2$
and $f \in L_{q,p}((0,T) \times \bR^d)$,
there exists a unique function $u \in W_{q,p}^{1,2}((0,T) \times \bR^d)$
satisfying $u(T,x) = 0$ and the equation \eqref{para_main_eq},
where, as in the usual Sobolev spaces, $u \in W_{q,p}^{1,2}((0,T) \times \bR^d)$
if $u$, $u_x$, $u_{xx}$, $u_t \in L_{q,p}((0,T) \times \bR^d)$.

As explained in \cite{Krylov_2007_mixed_VMO},
one of advantages of having the $L_{q,p}$-theory for parabolic equations
is that one can improve the regularity of solutions, for example, in $t$
using embedding theorems with a large $q$.
Indeed, one result in this paper is used 
in \cite{Doyoon:parabolic:2006} 
to estimate the H\"{o}lder continuity of a solution to a parabolic equation.
The same type of argument is also used in Lemma \ref{lemma03222007} of this paper.

The same or similar parabolic equations (or systems) 
in $L_{q,p}$-spaces (or {\em Sobolev spaces with mixed norms}) 
with {\em not necessarily continuous coefficients} have been dealt with, 
for example,
in \cite{MR2286441, MR1935919, Krylov_2007_mixed_VMO}.
However, the coefficients $a^{ij}$ there are assumed to be VMO in the spatial variables.
In fact, the coefficients in \cite{MR2286441} are VMO in $x \in \bR^d$, but independent of $t \in \bR$,
whereas the coefficients $a^{ij}$ in \cite{MR1935919}
are measurable functions of only $t$,
and the coefficients $a^{ij}$ in \cite{Krylov_2007_mixed_VMO} are 
measurable in $t$ and VMO in $x$.
Since we assume in this paper that the coefficients $a^{ij}$, except $a^{11}$, are measurable in $(t,x^1) \in \bR^2$ and VMO in $x' \in \bR^{d-1}$,
as far as coefficients $a^{ij}$, $i \ne 1$ or $j \ne 1$, are concerned,
the class of coefficients we consider is bigger than
those previously considered,
but since $a^{11}$ is measurable in $x^1 \in \bR$
but VMO in $(t,x') \in \bR \times \bR^{d-1}$,
we have to say that the class of coefficients $a^{ij}$ in this paper
is {\em different} from those, especially, in \cite{Krylov_2007_mixed_VMO}.
On the other hand, 
this paper is a continuation of the paper \cite{Krylov_2007_mixed_VMO} because the results and methods in that paper enable
us to deal with parabolic equations with the coefficients of this paper in the framework of {\em Sobolev spaces with mixed norms}.

Elliptic and parabolic equations in $L_p$-spaces
(not $L_{q,p}$-spaces) with coefficients as in this paper
are investigated in \cite{Doyoon&Krylov:article_meas:2005, Doyoon&Krylov:parabolic:2006,
Doyoon:parabolic:2006},
where one can find references about equations
with discontinuous coefficients as well as the non-solvability 
of equations with general measurable coefficients.
In addition to references stated in \cite{Doyoon&Krylov:article_meas:2005, Doyoon&Krylov:parabolic:2006,
Doyoon:parabolic:2006},
we refer the reader to papers \cite{MR578364, MR716741, MR1849379, MR1947468}
for examples of differential equations which do not have unique solutions in Sobolev spaces.
One can find in \cite{Softova2006} quasi-linear parabolic equations in mixed norms.
For more references about parabolic (or elliptic) equations in $L_p$ or $L_{q,p}$,
see \cite{MR2286441, Krylov_2007_mixed_VMO} and references therein.

This paper consist of two parts. In the first part we solve 
the equation \eqref{para_main_eq} in Sobolev spaces with mixed norms when the coefficients $a^{ij}$
are measurable in $x^1 \in \bR$ and VMO in $(t,x') \in \bR \times \bR^{d-1}$.
This result serves as one of main steps in \cite{Doyoon:parabolic:2006}.
Then using the results in \cite{Doyoon:parabolic:2006}
as well as in the first part of this paper,
we prove the main result of this paper.
The first part consists of section \ref{section04082007} and
\ref{section04102007};
the second part consists of section \ref{section04082007_01}
and \ref{section04092007}.
In section \ref{main_sec} we states the assumptions and the main result.

A few words about notation:
We denote by $(t,x)$ a point in $\bR^{d+1}$, i.e.,
$(t,x) = (t,x^1,x') \in \bR \times \bR^{d} = \bR^{d+1}$,
where $t \in \bR$, $x^1 \in \bR$, $x' \in \bR^{d-1}$,
and $x = (x^1,x') \in \bR^d$.
By $u_{x'}$ we mean, depending on the context,
one of $u_{x^j}$, $i = 2, \cdots,d$, or the whole collection
$\{ u_{x^2}, \cdots, u_{x^d} \}$.
As usual, $u_{x}$ represents one of $u_{x^i}$, $i = 1, \cdots, d$,
or the whole collection of $\{ u_{x^1}, \cdots, u_{x^d} \}$.
Thus $u_{xx'}$ is one of $u_{x^ix^j}$,
where $i \in \{1, \cdots, d\}$ and $j \in \{2, \cdots, d\}$,
or the collection of them.
For a function $u(t,x)$ defined on $\bR^{d+1}$ (or a subset in $\bR^{d+1}$), 
the average of $u$ over an open set $\cD \subset \bR^{d+1}$
is denoted by $\left(u\right)_{\cD}$, i.e.,
$$
\left(u\right)_{\cD}
= \frac{1}{|\cD|} \int_{\cD} u(t,x) \, dx \, dt
= \dashint_{\cD} u(t,x) \, dx \, dt,
$$
where $|\cD|$ is the $d+1$-dimensional volume of $\cD$.
Finally, various constants are denoted by $N$, 
their values may vary from one occurrence to another. We write $N(d, \delta, \dots)$ if $N$ depends only on $d$, $\delta$, $\dots$.

\vspace{1em}

\noindent{\sc Acknowledgement}: I would like to thank Nicolai V. Krylov for giving me an opportunity to learn and use his results in \cite{Krylov_2007_mixed_VMO}
before its publication.

\section{Main result}\label{main_sec}

We consider the parabolic equation \eqref{para_main_eq} 
with coefficients $a^{ij}$, $b^{i}$, and $c$ satisfying the following assumption.

\begin{assumption}\label{assum_01}
The coefficients $a^{ij}$, $b^{i}$, and $c$ are measurable functions defined on $\bR^{d+1}$, $a^{ij} = a^{ji}$.
There exist positive constants $\delta \in (0,1)$ and $K$ such that
$$
|b^{i}(t, x)| \le K, \quad |c(t, x)| \le K,
$$
$$
\delta |\vartheta|^2 \le \sum_{i,j=1}^{d} a^{ij}(t,x) \vartheta^i \vartheta^j
\le \delta^{-1} |\vartheta|^2
$$
for any $(t,x) \in \bR^{d+1}$ and $\vartheta \in \bR^d$.
\end{assumption}

Another assumption on the coefficients $a^{ij}$ is
that they are, in case $p \in (2, \infty)$, 
measurable in $(t,x^1) \in \bR^2$ and VMO in $x' \in \bR^{d-1}$
(the coefficient $a^{11}$ is measurable in $x^1 \in \bR$
and VMO in $(t,x') \in \bR^d$).
In case $p = 2$, the coefficients $a^{ij}$ are measurable functions of only $(t,x^1) \in \bR^2$,
but $a^{11}(t,x^1)$ is VMO in $t \in \bR$.
To state this assumption precisely, we introduce the following notation.
Let 
$$
B_r(x) = \{ y \in \bR^d: |x- y| < r \},
$$
$$
B'_r(x') = \{ y' \in \bR^{d-1}: |x'- y'| < r \},
$$
$$
Q_r(t,x) = (t, t+r^2) \times B_r(x),
\quad
\Gamma_r(t,x') = (t, t+r^2) \times B'_r(x'),
$$
$$
\Lambda_{r}(t,x)= (t, t+r^2) \times (x^1-r, x^1+r) \times B'_r(x').
$$
Set $B_r = B_r(0)$, $B'_r = B'_r(0)$, $Q_r = Q_r(0)$ and so on. 
By $|B'_r|$ we mean  the $d-1$-dimensional volume of $B'_r(0)$.
Denote 
$$
\text{osc}_{x'} \left( a^{ij}, \Lambda_r(t,x) \right) 
= r^{-3} |B'_r|^{-2} \int_{t}^{t+r^2}\int_{x^1-r}^{x^1+r} A^{ij}_{x'}(s,\tau)  \, d\tau \,  ds, 
$$
$$
\text{osc}_{(t,x')} \left( a^{ij}, \Lambda_r(t,x) \right) 
= r^{-5} |B'_r|^{-2}  \int_{x^1-r}^{x^1+r}  A^{ij}_{(t,x')}(\tau) \, d\tau, 
$$
where
$$
A^{ij}_{x'}(s,\tau)
= \int_{y', z' \in B'_r(x')}
|a^{ij}(s, \tau, y') - a^{ij}(s, \tau, z') | \, dy' \, dz',
$$
$$
A^{ij}_{(t,x')}(\tau)
= \int_{ (\sigma,y'), (\varrho,z') \in \Gamma_r(t,x')}
|a^{ij}(\sigma, \tau, y') - a^{ij}(\varrho, \tau, z') | \, dy' \, dz' \, d \sigma \, d \varrho.
$$
Also denote
$$
\cO_R^{\, x'}(a^{ij}) = \sup_{(t,x) \in \bR^{d+1}} \sup_{r \le R} \,\,\, \text{osc}_{x'} \left( a^{ij}, \Lambda_r(t,x) \right),
$$
$$
\cO_R^{(t,x')}(a^{ij}) = \sup_{(t,x) \in \bR^{d+1}} \sup_{r \le R} \,\,\, \text{osc}_{(t,x')} \left( a^{ij}, \Lambda_r(t,x) \right).
$$
Finally set
$$
a_R^{\#} =\cO_R^{(t,x')}(a^{11}) + \sum_{i \ne 1 \, \text{or} \, j \ne 1} \cO_R^{\, x'}(a^{ij}).
$$

\begin{assumption}\label{assum_02}
There is a continuous function $\omega(t)$ defined on $[0,\infty)$ such that $\omega(0) = 0$ and $a_R^{\#} \le \omega(R)$ for all $R \in [0,\infty)$.
\end{assumption}

By $W_{q,p}^{1,2}((S,T) \times \bR^{d})$
we mean the collection of all functions defined on $(S,T) \times \bR^{d}$,
$-\infty \le S < T \le \infty$,
such that 
$$
\| u \|_{W_{q,p}^{1,2}((S,T) \times \bR^{d})}
:= \|u \|_{L_{q,p}((S,T) \times \bR^{d})} 
+ \| u_{x} \|_{L_{q,p}((S,T) \times \bR^{d})} 
$$
$$
+ \| u_{x x} \|_{L_{q,p}((S,T) \times \bR^{d})}
+ \| u_t \|_{L_{q,p}((S,T) \times \bR^{d})} < \infty.
$$
We say $u \in \WO{1,2}{q,p}((S,T) \times \bR^d)$
if $u \in  W_{q,p}^{1,2}((S,T) \times \bR^d)$
and $u(T,x) = 0$.
Throughout the paper, we set
$$
L_{q,p} := L_{q,p} (\bR \times \bR^d),
\quad
W_{q,p}^{1,2} : = W_{q,p}^{1,2}(\bR \times \bR^{d}).
$$
As usual, 
$$
L_p((S, T) \times \bR^d)
= L_{p,p}((S, T) \times \bR^d),
$$
$$
W_p^{1,2}((S,T) \times \bR^d)
= W_{p,p}^{1,2}((S,T) \times \bR^d).
$$

We denote the differential operator by $L$, that is,
$$
Lu = u_t + a^{ij} u_{x^i x^j} + b^{i} u_{x^i} + c u.
$$
The following is our main result.

\begin{theorem} \label{theorem03192007}
Let $q > p \ge 2$, $0 < T < \infty$,
and the coefficients of $L$ satisfy
Assumption \ref{assum_01} and \ref{assum_02}.
In addition, 
we assume that, in case $p=2$,
$a^{ij}$ are independent of $x' \in \bR^{d-1}$.
Then for any $f \in L_{q,p}((0,T) \times \bR^d)$, 
there exists a unique $u \in \WO{1,2}{q,p}((0,T) \times \bR^d)$ such that 
$Lu = f$ in $(0,T) \times \bR^d$.
Furthermore, there is a constant $N$, depending only on $d$, $\delta$, $K$, $p$, $q$, $T$, and $\omega$,
such that, for any $u \in \WO{1,2}{q,p}((0,T) \times \bR^d)$,
$$
\| u \|_{W_{q,p}^{1,2}((0,T) \times \bR^d)}
\le N \| Lu \|_{L_{q,p}((0,T) \times \bR^d)}.
$$
\end{theorem}

\section{Equations with $a^{ij}$ measurable in $x^1 \in \bR$ and {VMO} in $(t,x') \in \bR^d$}								\label{section04082007}

In this section
we suppose that the coefficients $a^{ij}$ are
measurable in $x^1 \in \bR$ and VMO in $(t,x') \in \bR^{d}$.
To state this assumption, set
$$
a_R^{\#(t,x')} = \sum_{i, j = 1}^d \cO_R^{(t,x')}(a^{ij}).
$$

\begin{assumption}\label{assum_03}
There is a continuous function $\omega(t)$ defined on $[0,\infty)$ such that $\omega(0) = 0$ and $a_R^{\#(t,x')} \le \omega(R)$ for all $R \in [0,\infty)$.
\end{assumption}

The following is the main result of this section.

\begin{theorem}							\label{theorem04022007}
Let $q > p \ge 2$, $0 < T < \infty$,
and the coefficients of $L$ satisfy
Assumption \ref{assum_01} and \ref{assum_03}.
In addition, 
we assume that, in case $p=2$,
$a^{ij}$ are independent of $x' \in \bR^{d-1}$.
Then for any $f \in L_{q,p}((0,T) \times \bR^d)$, 
there exists a unique $u \in \WO{1,2}{q,p}((0,T) \times \bR^d)$ such that 
$Lu = f$ in $(0,T) \times \bR^d$.
Furthermore, there is a constant $N$, depending only on $d$, $\delta$, $K$, $p$, $q$, $T$, and $\omega$,
such that, for any $u \in \WO{1,2}{q,p}((0,T) \times \bR^d)$,
$$
\| u \|_{W_{q,p}^{1,2}((0,T) \times \bR^d)}
\le N \| Lu \|_{L_{q,p}((0,T) \times \bR^d)}.
$$
\end{theorem}

This theorem is proved in the next section after presenting some preliminary results.
Throughout this section, we set
$$
\cL_{\lambda} u = u_t + a^{ij}(x^1) u_{x^i x^j} - \lambda u,
$$
where $\lambda \ge 0$ and $a^{ij}$ are measurable functions of only $x^1 \in \bR$ satisfying
Assumption \ref{assum_01}.

We start with a theorem which can be derived from results in \cite{Doyoon&Krylov:parabolic:2006}.

\begin{theorem}							\label{theorem04012007}
Let $p \ge 2$ and $T \in [-\infty,\infty)$.
For any $\lambda>0$ and $f\in L_{p}((T,\infty) \times \bR^{d})$,
there exists a unique solution $u \in W^{1,2}_{p}((T,\infty) \times \bR^{d})$ 
to the equation
$\cL_{\lambda}u = f$. 
Furthermore, there is a constant $N=N(d, p, \delta)$
such that, for any $\lambda \ge 0$ and
$u \in W^{1,2}_{p}((T,\infty) \times \bR^{d})$,
we have
$$
\|u_{t}\|_{L_{p}((T,\infty) \times \bR^{d})}
+\|u_{xx}\|_{L_{p}((T,\infty) \times \bR^{d})}
+\sqrt{\lambda}\|u_{x}\|_{L_{p}((T,\infty) \times \bR^{d})}
$$
$$
+\lambda\|u\|_{L_{p}((T,\infty) \times \bR^{d})}
\le N\|\cL_{\lambda} u\|_{L_{p}((T,\infty) \times \bR^{d})}.
$$
\end{theorem}

If $T = -\infty$,
this theorem is obtained 
from Theorem 3.2 in \cite{Doyoon&Krylov:parabolic:2006} for $p = 2$
and Lemma 5.3\footnote{In fact, Lemma 5.3 in \cite{Doyoon&Krylov:parabolic:2006}
says that the estimate in Theorem \ref{theorem04012007}
holds for all $\lambda \ge \lambda_0$, where $\lambda_0 \ge 0$,
that is, $\lambda_0$ may not be $0$.
However, since the coefficients $a^{ij}$ of $\cL_{\lambda}$
are measurable functions of only $x^1 \in \bR$
and $b^i = c = 0$,
it can be proved, using a dilation argument,
that $\lambda_0 = 0$ in our case.}
in \cite{Doyoon&Krylov:parabolic:2006} for $p > 2$.
For the case $T \in (-\infty, \infty)$, we use the case $T = -\infty$ and the argument following Corollary 5.14 in \cite{Krylov_2007_mixed_VMO}.

The following three lemmas are 
$L_p$-versions of Lemma 4.2, 4.3, and 4.4 in \cite{Doyoon&Krylov:parabolic:2006}.
Since the estimate in Theorem \ref{theorem04012007}
is available, their proofs can be done
by repeating the proofs of
Lemma 4.2, 4.3, and 4.4 in \cite{Doyoon&Krylov:parabolic:2006}
with $p$ in place of $2$.

\begin{lemma} \label{lemma04012007_01}
Let $p \in [2, \infty)$.
For any $u \in W_{p, \text{loc}}^{1,2}(\bR^{d+1})$,
we have
$$
\|u_t\|_{L_p(Q_r)}
+ \| u_{xx} \|_{L_p(Q_r)}
+ \| u_{x} \|_{L_p(Q_r)}
\le N \left( \| \cL_0 u \|_{L_p(Q_R)} + \| u \|_{L_p(Q_R)} \right),
$$
where $0 < r < R < \infty$ and $N = N(d, p, \delta, r, R)$.
\end{lemma}

\begin{lemma} \label{lemma04012007_02}
Let $p \in [2, \infty)$, $0 < r < R < \infty$, and
$\gamma = (\gamma^1, \cdots, \gamma^d)$ be a multi-index such that $\gamma^1 = 0, 1, 2$. 
If $v \in C_{\text{loc}}^{\infty}(\bR^{d+1})$ is a function 
such that $\cL_0 v = 0$ in $Q_R$, then 
$$
\int_{Q_r} | D_t^{m}D_{x}^{\gamma} v |^p \, dx \, dt
\le N 
\int_{Q_R} |v|^p \, dx \, dt,
$$
where $m$ is a nonnegative integer and $N = N(d, p, \delta, \gamma, m, r, R)$.
\end{lemma}

\begin{lemma} \label{lemma04012007}
Let $p \ge 2$
and $v \in C_{\text{loc}}^{\infty}(\bR^{d+1})$
be a function such that
$\cL_{0} v = 0 $ in $Q_{4}$.
Then
$$
\sup_{Q_1}|v_{tt}|
+ \sup_{Q_1}|v_{tx}|
+ \sup_{Q_1}|v_{txx'}|
+ \sup_{Q_1}|v_{xxx'}|
\le N \| v \|_{L_p(Q_4)},
$$
where $N = N(d, p, \delta)$.
\end{lemma}

The proofs of the lemmas and theorem below are 
almost identical to those in \cite{Krylov_2007_mixed_VMO},
specifically, proofs of Lemma 5.9, Theorem 5.10,
and Theorem 5.1 in \cite{Krylov_2007_mixed_VMO}.
Basically, one can follow the steps in the proofs there
using the above lemmas.
However, rather than referring to \cite{Krylov_2007_mixed_VMO},
we give here complete proofs
to provide the details of our case.

\begin{lemma} \label{lemma04012007_03}
Let $p \ge 2$.
For every $v \in C_{\text{loc}}^{\infty}(\bR^{d+1})$
such that $\cL_{\lambda} v = 0$ in $Q_{4}$,
we have
$$
\sup_{Q_1}|v_{tt}|
+ \sup_{Q_1}|v_{tx}|
+ \sup_{Q_1}|v_{txx'}|
+ \sup_{Q_1}|v_{xxx'}|
$$
$$
\le N(d, p, \delta) 
\left( \|v_{xx}\|_{L_p(Q_4)} + \|v_t\|_{L_p(Q_4)} + \sqrt{\lambda} \|v_x\|_{L_p(Q_4)} \right).
$$
\end{lemma}

\begin{proof}
We first note that, in case $\lambda = 0$,
by Lemma \ref{lemma04012007}
$$
I := \sup_{Q_1}|v_{tt}|
+ \sup_{Q_1}|v_{tx}|
+ \sup_{Q_1}|v_{txx'}|
+ \sup_{Q_1}|v_{xxx'}|
\le N \| v \|_{L_p(Q_4)},
$$
The function $u:=v - \left(v\right)_{Q_4} - x^i \left(v_{x^i}\right)_{Q_4}$
can replace $v$ in the above inequality
since $\cL_0 u = 0$ in $Q_4$.
This together with the fact that
$D_t^m D_{x}^{\gamma} v = D_t^m D_{x}^{\gamma} u$
for $m \ge 1$ or $|\gamma| \ge 2$
gives us
$$
I
\le N \| v - \left(v\right)_{Q_4} - x^i \left(v_{x^i}\right)_{Q_4} \|_{L_p(Q_4)}.
$$
This and Lemma 5.4 in \cite{Krylov_2007_mixed_VMO} prove the inequality in the lemma 
for $\lambda = 0$.

In case $\lambda > 0$, 
let $\mathbf{v}(t,\mathbf{x}) = \mathbf{v}(t,x,\xi)$ be a function on $\bR^{d+2}$
defined by
$$
\mathbf{v}(t,\mathbf{x}) = v(t,x) \cos(\sqrt{\lambda} \, \xi),
$$
where $(t, \mathbf{x}) = (t, x, \xi)$, $\xi \in \bR$,
and
$$
\mathbf{Q}_r = (0,r^2) \times \{\mathbf{x} \in \bR^{d+1}: |\mathbf{x}| < r\}.
$$
Observe that
$$
D_t^m D_x^{\gamma} v(t,x)
= D_t^m D_x^{\gamma} \mathbf{v}(t,x,0),
$$
where $m$ is a non-negative integer and $\gamma$ is a multi-index
with respect to $x \in \bR^d$.
Thus
$$
\sup_{Q_1}|D_t^m D_x^{\gamma} v|
\le \sup_{\mathbf{Q}_1}|D_t^m D_x^{\gamma} \mathbf{v}|.
$$
In addition,
$$
\cL_0 \mathbf{v} + \mathbf{v}_{\xi \xi} = 0
\quad
\text{in}
\quad
\mathbf{Q}_{4}.
$$
Hence by the above reasoning for the case $\lambda = 0$  
we have
\begin{equation}							\label{04012007_01}
I
\le N \left( \| \mathbf{v}_{\mathbf{x}\mathbf{x}} \|_{L_p(\mathbf{Q}_4)}
+ \| \mathbf{v}_{t} \|_{L_p(\mathbf{Q}_4)}
\right).
\end{equation}
We see that $\mathbf{v}_{\mathbf{x}\mathbf{x}}$ is
$$
v_{xx} \cos ( \sqrt{\lambda} \, \xi),
\quad
- \sqrt{\lambda} v_{x} \sin ( \sqrt{\lambda} \, \xi),
\quad
\text{or}
\quad
- \lambda v \cos ( \sqrt{\lambda} \, \xi).
$$
Therefore, the right-hand side of the inequality \eqref{04012007_01}
is not greater than a constant times 
$$
\| v_{xx} \|_{L_p(Q_4)}
+ \| v_{t} \|_{L_p(Q_4)} + \sqrt{\lambda} \| v_{x} \|_{L_p(Q_4)}
+ \lambda \| v \|_{L_p(Q_4)}.
$$
This is bounded by the right-hand side of the inequality in the lemma
(note that $\lambda v = \cL_0 v$ in $Q_4$).
The lemma is proved.
~\end{proof}

\begin{lemma} \label{lemma04012007_04}
Let $p \ge 2$, $\lambda \ge 0$, $\kappa \ge 4$, and $r \in (0,\infty)$.
Let $v \in C_{\text{loc}}^{\infty}(\bR^{d+1})$ be 
such that
$\cL_{\lambda}v = 0$ in $Q_{\kappa r}$.
Then there is a constant $N$, depending only on $d$, $p$, and $\delta$,
such that
\begin{multline}							\label{04012007_02}
\dashint_{Q_r} | v_{t}(t,x) - \left( v_{t} \right)_{Q_r} |^p
\, dx \, dt
+ \dashint_{Q_r} | v_{xx'}(t,x) - \left( v_{xx'} \right)_{Q_r} |^p
\, dx \, dt \\
\le N \kappa^{-p} \left( |v_{xx}|^p + |v_{t}|^p + \lambda^{p/2} |v_{x}|^p \right)_{Q_{\kappa r}}.
\end{multline}
\end{lemma}

\begin{proof}
We first show that that the inequality \eqref{04012007_02}
follows from the case with $r = 1$.
To see this, for a given $v \in C_{\text{loc}}^{\infty}(\bR^{d+1})$
such that $\cL_{\lambda} v = 0$ in $Q_{\kappa r}$, $r > 0$, 
we set $\hat{v}(t,x) = v(r^2 t, r x)$.
Then $\hat{v}$ satisfies 
$$
\hat{\cL}_{r^2 \lambda} \hat{v}(t,x) := 
\left(\frac{\partial}{\partial t} + a^{ij}(r x^1)
\frac{\partial^2}{\partial x^i \partial x^j} - r^2 \lambda \right) \hat{v}(t,x)
$$
$$
= r^2 \left( \cL_{\lambda} v \right)(r^2 t, r x)
= 0
\quad
\text{in}
\quad
Q_{\kappa}.
$$
Note that the coefficients $a^{ij}(r x^1)$ satisfy 
Assumption \ref{assum_01} with the same $\delta$.
Thus, if the inequality \eqref{04012007_02} holds true for $r = 1$,
we have
$$
\dashint_{Q_1} | \hat{v}_{t}(t,x) - \left( \hat{v}_{t} \right)_{Q_1} |^p
\, dx \, dt
+
\dashint_{Q_1} | \hat{v}_{xx'}(t,x) - \left( \hat{v}_{xx'} \right)_{Q_1} |^p
\, dx \, dt
$$
$$
\le N \kappa^{-p} \left( |\hat{v}_{xx}|^p + |\hat{v}_{t}|^p + (r^2\lambda)^{p/2} |\hat{v}_{x}|^p \right)_{Q_{\kappa}}.
$$
This proves the inequality \eqref{04012007_02} for $r > 0$ since
$$
\dashint_{Q_1} | \hat{v}_{t}(t,x) - \left( \hat{v}_{t} \right)_{Q_1} |^p
\, dx \, dt
= r^{2p} \dashint_{Q_r} | v_{t}(t,x) - \left( v_{t} \right)_{Q_r} |^p
\, dx \, dt,
$$
$$
\dashint_{Q_1} | \hat{v}_{xx'}(t,x) - \left( \hat{v}_{xx'} \right)_{Q_1} |^p
\, dx \, dt
= r^{2p} \dashint_{Q_r} | v_{xx'}(t,x) - \left( v_{xx'} \right)_{Q_r} |^p
\, dx \, dt
$$
$$
\left( |\hat{v}_{xx}|^p + |\hat{v}_{t}|^p + (r^2\lambda)^{p/2} |\hat{v}_{x}|^p \right)_{Q_{\kappa}}
= r^{2p}
\left( |v_{xx}|^p + |v_{t}|^p + \lambda^{p/2} |v_{x}|^p \right)_{Q_{\kappa r}}.
$$

Now we prove the inequality \eqref{04012007_02} for $r = 1$.
For $v \in C_{\text{loc}}^{\infty}(\bR^{d+1})$ such that $\cL_{\lambda} v = 0$
in $Q_{\kappa}$, $\kappa \ge 4$,
set
$$
\check{v}(t,x) = v \left(\left(\frac{\kappa}{4}\right)^2 \! t, \, \frac{\kappa}{4} x \right),
\quad
\check{a}^{ij}(x^1) = a^{ij}(\kappa x^1/4).
$$
Then
$$
\check{\cL}_{(\frac{\kappa}{4})^2 \lambda} \check{v}(t,x)
:=
\left(\frac{\partial}{\partial t} + \check{a}^{ij}(x^1)
\frac{\partial^2}{\partial x^i \partial x^j} - \left(\frac{\kappa}{4}\right)^2 \lambda \right) \check{v}(t,x)
$$
$$
= \left(\frac{\kappa}{4}\right)^2 \left( \cL_{\lambda} v \right)\left(\left(\frac{\kappa}{4}\right)^2 \! t, \, \frac{\kappa}{4} x \right)
= 0
\quad
\text{in}
\quad
Q_{4}.
$$
Thus by Lemma \ref{lemma04012007_03},
it follows that
\begin{equation}							\label{04012007_001}
\check{I} \le N \left( \| \check{v}_{xx} \|_{L_p(Q_4)}
+ \| \check{v}_{t} \|_{L_p(Q_4)}
+ \frac{\kappa}{4}\sqrt{\lambda}\| \check{v}_{x} \|_{L_p(Q_4)} \right),
\end{equation}
where
$$
\check{I} = \sup_{Q_1}|\check{v}_{tt}|
+ \sup_{Q_1}|\check{v}_{tx}|
+ \sup_{Q_1}|\check{v}_{txx'}|
+ \sup_{Q_1}|\check{v}_{xxx'}|.
$$
Note that 
$$
\left(4/\kappa\right)^3 \check{I} = 
\left(\kappa/4\right)\left(\sup_{Q_{\kappa/4}}|v_{tt}|
+ \sup_{Q_{\kappa/4}}|v_{txx'}|\right)
+ \sup_{Q_{\kappa/4}}|v_{tx}|
+ \sup_{Q_{\kappa/4}}|v_{xxx'}|.
$$
Using this, the inequality \eqref{04012007_001},
and $\kappa \ge 4$,
we have
$$
\dashint_{Q_1} | v_{t}(t,x) - \left( v_{t} \right)_{Q_1} |^p
\, dx \, dt
+ \dashint_{Q_1} | v_{xx'}(t,x) - \left( v_{xx'} \right)_{Q_1} |^p
\, dx \, dt
$$
$$
\le N 
\left(
\sup_{Q_{\kappa/4}}|v_{tt}|
+ \sup_{Q_{\kappa/4}}|v_{tx}|
+ \sup_{Q_{\kappa/4}}|v_{txx'}|
+ \sup_{Q_{\kappa/4}}|v_{xxx'}|
\right)^p
\le N
\kappa^{-3p} {\check{I}}^{p}
$$
$$
\le  N \kappa^{-3p} \left( \| \check{v}_{xx} \|^p_{L_p(Q_4)}
+ \| \check{v}_{t} \|^p_{L_p(Q_4)}
+ \kappa^{p}\lambda^{p/2}\| \check{v}_{x} \|^p_{L_p(Q_4)} \right)
$$
$$
= N \kappa^{-p} \left( |v_{xx} |^p + |v_{t}|^p 
+ \lambda^{p/2} |v_{x}|^p \right)_{L_p(Q_{\kappa})}.
$$
This finishes the proof.
~\end{proof}

\begin{theorem} \label{theorem04022007_01}
Let $p \ge 2$. 
Then there is a constant $N$, depending only on $d$, $p$, and $\delta$,
such that, for any $u \in W_p^{1,2}(\bR^{d+1})$,
$r \in (0, \infty)$, and $\kappa \ge 8$,
$$
\dashint_{Q_r} | u_{t}(t,x) - \left( u_{t} \right)_{Q_r} |^p
\, dx \, dt
+
\dashint_{Q_r} | u_{xx'}(t,x) - \left( u_{xx'} \right)_{Q_r} |^p
\, dx \, dt
$$
$$
\le N \kappa^{d+2} \left( |\cL_0 u|^p \right)_{Q_{\kappa r}}
+ N \kappa^{-p} \left( |u_{xx}|^p \right)_{Q_{\kappa r}}.
$$
\end{theorem}

\begin{proof}
Since $C_0^{\infty}(\bR^{d+1})$ is dense in $W_p^{1,2}(\bR^{d+1})$,
it is enough to have $u \in C_0^{\infty}(\bR^{d+1})$.
In addition, we can assume that $a^{ij}(x^1)$ are infinitely differentiable.
Take a $\lambda > 0$
and, for $u \in C_0^{\infty}(\bR^{d+1})$,
let
$$
f:= f_{\lambda} = \cL_{\lambda} u.
$$
We see $f \in C_{0}^{\infty}(\bR^{d+1})$.
For given $r > 0$ and $\kappa \ge 8$,
let $\eta \in C_0^{\infty}(\bR^{d+1})$
be a function such that $\eta = 1$ on $Q_{\kappa r/2}$
and $\eta = 0$ outside $\left(-(\kappa r)^2, (\kappa r)^2\right) \times B_{\kappa r}$.
Also let
$$
g := f \eta,
\quad
h := f(1- \eta).
$$
Then by Theorem \ref{theorem04012007} there exists a unique solution $v \in W_p^{1,2}(\bR^{d+1})$
(note that $\lambda > 0$)
to the equation 
$$
\cL_{\lambda} v = h.
$$
From the classical theory
we see that the function $v$ is infinitely differentiable.
Moreover, since $\cL_{\lambda} v = h = 0$ in $Q_{\kappa r/2}$
and $\kappa/2 \ge 4$, by Lemma \ref{lemma04012007_04},
we have
$$
\dashint_{Q_r} | v_{t}(t,x) - \left( v_{t} \right)_{Q_r} |^p
\, dx \, dt
+
\dashint_{Q_r} | v_{xx'}(t,x) - \left( v_{xx'} \right)_{Q_r} |^p
\, dx \, dt
$$
$$
\le N \kappa^{-p} \left( |v_{xx}|^p + |v_{t}|^p + \lambda^{p/2} |v_{x}|^p \right)_{Q_{\kappa r/2}}
$$
$$
\le N \kappa^{-p} \left( |v_{xx}|^p + |v_{t}|^p + \lambda^{p/2} |v_{x}|^p \right)_{Q_{\kappa r}}.
$$
Set $w := u - v \in W_{p}^{1,2}(\bR^{d+1})$.
Then from the above inequality it follows that
$$
\dashint_{Q_r} | u_{xx'}(t,x) - \left( u_{xx'} \right)_{Q_r} |^p
\, dx \, dt
$$
$$
\le 2^p \dashint_{Q_r} | w_{xx'}(t,x) - \left( w_{xx'} \right)_{Q_r} |^p
\, dx \, dt
+ 2^p \dashint_{Q_r} | v_{xx'}(t,x) - \left( v_{xx'} \right)_{Q_r} |^p
\, dx \, dt
$$
$$
\le N \left( |w_{xx'}|^p \right)_{Q_r}
+ N \kappa^{-p} \left( |v_{xx}|^p + |v_{t}|^p + \lambda^{p/2} |v_{x}|^p \right)_{Q_{\kappa r}}.
$$
Similar inequalities are possible with $u_t$ in place of $u_{xx'}$.
Thus we have
\begin{multline}							\label{04012007_002}
\dashint_{Q_r} | u_{t}(t,x) - \left( u_{t} \right)_{Q_r} |^p
\, dx \, dt
+ \dashint_{Q_r} | u_{xx'}(t,x) - \left( u_{xx'} \right)_{Q_r} |^p
\, dx \, dt \\
\le N \left( |w_{t}|^p + |w_{xx}|^p\right)_{Q_r}
+ N \kappa^{-p} \left( |v_{xx}|^p + |v_{t}|^p + \lambda^{p/2} |v_{x}|^p \right)_{Q_{\kappa r}}.
\end{multline}

Now we observe that
$$
\cL_{\lambda} w = \cL_{\lambda}(u-v)  = f - h = g
$$
and, by Theorem \ref{theorem04012007},
$$
\int_{Q_r} |w_{t}|^p \, dx \, dt
+ \int_{Q_r} |w_{xx}|^p \, dx \, dt
$$
$$
\le \|w_{t}\|^p_{L_{p}((0,\infty)\times\bR^d)}
+\|w_{xx}\|^p_{L_{p}((0,\infty)\times\bR^d)}
+\lambda^{p/2}\|w_{x}\|^p_{L_{p}((0,\infty)\times\bR^d)}
$$
$$
\le N \|g\|^p_{L_{p}((0,\infty)\times\bR^d)}
= N \int_{Q_{\kappa r}} |g|^p \, dx \, dt
\le N \int_{Q_{\kappa r}} |f|^p \, dx \, dt.
$$
From this we see that
$$
\left(|w_{t}|^p\right)_{Q_r} + \left(|w_{xx}|^p\right)_{Q_r}
\le N \kappa^{d+2} \left( |f|^p \right)_{Q_{\kappa r}},
$$
$$
\left( |w_{xx}|^p + |w_{t}|^p + \lambda^{p/2} |w_{x}|^p \right)_{Q_{\kappa r}}
\le N \left( |f|^p \right)_{Q_{\kappa r}}.
$$

Now we use the these inequalities as well as the inequality \eqref{04012007_002}.
We also use the fact $u = w + v$ and $\kappa \ge 8$.
Then we obtain
$$
\dashint_{Q_r} | u_{t}(t,x) - \left( u_{t} \right)_{Q_r} |^p
\, dx \, dt
+ \dashint_{Q_r} | u_{xx'}(t,x) - \left( u_{xx'} \right)_{Q_r} |^p
\, dx \, dt
$$
$$
\le N \kappa^{d+2} \left( |f|^p \right)_{Q_{\kappa r}}
+ N \kappa^{-p} \left( |w_{xx}|^p + |w_{t}|^p + \lambda^{p/2} |w_{x}|^p \right)_{Q_{\kappa r}}
$$
$$
+ N \kappa^{-p} \left( |u_{xx}|^p + |u_{t}|^p + \lambda^{p/2} |u_{x}|^p \right)_{Q_{\kappa r}}
$$
$$
\le N \kappa^{d+2} \left( |f|^p \right)_{Q_{\kappa r}}
+ N \kappa^{-p} \left( |u_{xx}|^p + |u_{t}|^p + \lambda^{p/2} |u_{x}|^p \right)_{Q_{\kappa r}}.
$$
To complete the proof, we use the fact
that $u_t = f + \lambda u - a^{ij} u_{x^i x^j}$,
and then let $\lambda \searrow 0$.
~\end{proof}

\section{Proof of Theorem \ref{theorem04022007}}		\label{section04102007}

Recall $Lu = u_t + a^{ij}(t,x) u_{x^i x^j} + b^{i} u_{x^i} + c u$,
where coefficients $a^{ij}$, $b^{i}$, and $c$ satisfy
Assumption \ref{assum_01} and \ref{assum_03}.
Set 
$$
L_0 u = u_t + a^{ij}(t,x) u_{x^i x^j}.
$$

\begin{lemma}							\label{lemma04022007}
Let $p > q \ge 2$, and $r \in (0,1]$.
Assume that $v \in W_{p,\text{loc}}^{1,2}(\bR^{d+1})$
satisfies $L_0 v = 0$ in $Q_{2r}$.
Then
$$
\left( |v_{xx}|^p \right)_{Q_r}^{1/p}
\le
N \left( |v_{xx}|^2 \right)^{1/2}_{Q_{2r}}
\le
N \left( |v_{xx}|^q \right)_{Q_{2r}}^{1/q},
$$
where $N$ depends only on $d$, $p$, $\delta$, and the function $\omega$.
\end{lemma}

\begin{proof}
This lemma is almost the same as 
Corollary 6.4 in \cite{Krylov_2007_mixed_VMO}
if $L_0$ is replaced by the operator used there.
In our case, we can repeat the argument in Corollary 6.4 of \cite{Krylov_2007_mixed_VMO}
if we have the estimate 
$$
\| u_{xx} \|_{L_p(Q_r)}
\le N
\left( \|L_0 u\|_{L_p(Q_{\kappa r})}
+ r^{-1}\|u_x\|_{L_p(Q_{\kappa r})}
+ r^{-2}\|u\|_{L_p(Q_{\kappa r})} \right)
$$
for $p \in (2, \infty)$ and $u \in W_{p, \text{loc}}^{1,2}(\bR^{d+1})$,
where $r \in (0,1]$, $\kappa \in (1, \infty)$,
and $N$ depends only on $d$, $p$, $\delta$, $\kappa$, and the function $\omega$.
This is obtained using
Theorem 2.5 in \cite{Doyoon&Krylov:parabolic:2006}
and the argument in the proof of Lemma 6.3 of \cite{Krylov_2007_mixed_VMO}.
~\end{proof}

The following theorem is proved in the same way as Lemma 3.1 in \cite{Krylov_2007_mixed_VMO}.
Because of the difference between our operator $L$ (or $L_0$) 
and the operator defined in \cite{Krylov_2007_mixed_VMO},
we give a complete proof here.

\begin{theorem}								\label{theorem04082007}
Let $p \ge 2$.
In case $p = 2$,
the coefficients $a^{ij}$ of $L_0$
are assumed to be independent of $x' \in \bR^{d-1}$.
Then there exists a constant $N$, depending on $d$, $p$, $\delta$, and 
the function $\omega$, such that,
for any $u \in C_0^{\infty}(\bR^{d+1})$, $\kappa \ge 16$,
and $r \in (0, 1/\kappa]$,
we have
$$
\dashint_{Q_r}
| u_{t}(t,x) - \left( u_{t} \right)_{Q_r} |^p \, dx \, dt
+ \dashint_{Q_r}
| u_{xx'}(t,x) - \left( u_{xx'} \right)_{Q_r} |^p \, dx \, dt
$$
$$
\le N \kappa^{d+2} \left( |L_0 u|^p \right)_{Q_{\kappa r}}
+ N \left( \kappa^{-p} + \kappa^{d+2} 
\hat{a}^{1/2} \right) \left( |u_{xx}|^p \right)_{Q_{\kappa r}},
$$
where $\hat{a} = a_{\kappa r}^{\#(t,x')}$.
\end{theorem}

\begin{proof}
For given $u \in C_{0}^{\infty}(\bR^{d+1})$,
$\kappa \ge 16$, and $r \in (0, 1/\kappa]$,
find a unique function $\tilde{w} \in W_p^{1,2}((-3,4) \times \bR^d)$ satisfying
$\tilde{w}(4,x) = 0$ and
$$
L_0 \tilde{w} = f I_{Q_{\kappa r}},
$$
where $f : = L_0 u$.
This is possible by Theorem 2.2 and 2.5 in \cite{Doyoon&Krylov:parabolic:2006}.
In fact, $\tilde{w} \in W_q^{1,2}((-3,4) \times \bR^d)$ for all $q \in (2, \infty)\footnote{We may not be able to have $q = 2$ if $p \ne 2$
and the coefficients $a^{ij}$ are not independent of $x' \in \bR^{d-1}$.}$
because $f I_{Q_{\kappa r}} \in L_q((-3,4) \times \bR^d)$ for all $q > 2$.
Let
$$
w(t,x) = \eta(t) \tilde{w}(t,x),
$$
where
$\eta(t)$ is an infinitely differentiable function defined on $\bR$
such that
$$
\eta(t) = 1, \quad -1 \le t \le 2,
\quad
\eta(t) = 0, \quad t \le -2 \quad \text{or} \quad t \ge 3.
$$
We see that $w \in W_p^{1,2}(\bR^{d+1})$
and, in addition, $w \in W_q^{1,2}(\bR^{d+1})$ for all $q \in (2, \infty)$.
From the estimates from Theorem 2.2 and 2.5 in \cite{Doyoon&Krylov:parabolic:2006}
we have
$$
\int_{Q_{\kappa r}} \left( |w_{t}|^p + |w_{xx}|^p \right) \, dx \, dt
\le \int_{(-3,4) \times \bR^d} \left( |\tilde{w}_{t}|^p + |\tilde{w}_{xx}|^p \right) \, dx \, dt
$$
$$
\le N\int_{Q_{\kappa r}} |f|^p \, dx \, dt,
$$
where $N$ depends only on $d$, $\delta$, $p$, and $\omega$
(it also depends on the time interval, but the time interval here is fixed as $(-3,4)$).
Thus
\begin{equation}							\label{04022007_02}
\left( |w_{t}|^p + |w_{xx}|^p \right)_{Q_{\kappa r}}
\le N \left( |f|^p \right)_{Q_{\kappa r}},
\end{equation}
\begin{equation}							\label{04022007_03}
\left( |w_{t}|^p + |w_{xx}|^p \right)_{Q_{r}}
\le N \kappa^{d+2} \left( |f|^p \right)_{Q_{\kappa r}}.
\end{equation}
where $N = N(d,\delta,p,\omega)$.

Now we set
$$
v = u - w.
$$
Then $v \in W_{p}^{1,2}(\bR^{d+1})$,
$v \in W_{q}^{1,2}(\bR^{d+1})$, $q \in (2, \infty)$,
and 
$$
L_0 v = 0 \quad \text{in} \quad Q_{\kappa r}.
$$

Let
$$
\bar{L}_0 = \frac{\partial}{\partial t}
+ \bar{a}^{ij}(x^1) \frac{\partial^2}{\partial x^i \partial x^j},
$$
where
$$
\bar{a}^{ij}(x^1) = 
\dashint_{\Gamma_{\kappa r/2}} a^{ij}(s,x^1,y') \, dy'\, ds.
$$
Since $v \in W_{p}^{1,2}(\bR^{d+1})$
and $\kappa/2 \ge 8$,
by Theorem \ref{theorem04022007_01}
applied to the operator $\bar{L}_0$,
we have
$$
\dashint_{Q_r} | v_{t}(t,x) - \left( v_{t} \right)_{Q_r} |^p
\, dx \, dt
+
\dashint_{Q_r} | v_{xx'}(t,x) - \left( v_{xx'} \right)_{Q_r} |^p
\, dx \, dt
$$
$$
\le N \kappa^{d+2} \left( |\bar{L}_0 v|^p \right)_{Q_{\kappa r/2}}
+ N \kappa^{-p} \left( |v_{xx}|^p \right)_{Q_{\kappa r/2}}.
$$
Using the fact that $L_0 v = 0$ in $Q_{\kappa r}$, 
we have
$$
\left( |\bar{L}_0 v|^p \right)_{Q_{\kappa r/2}}
= \dashint_{Q_{\kappa r/2}}
| \left(\bar{a}^{ij}(x^1) - a^{ij}(t,x) \right) v_{x^i x^j} |^p
\, dx \, dt
$$
$$
\le
\left( \dashint_{Q_{\kappa r/2}} |\bar{a}^{ij}(x^1) - a^{ij}(t,x)|^{2p} \, dx \, dt \right)^{1/2}
\left( \dashint_{Q_{\kappa r/2}} |v_{x^i x^j}|^{2p} \, dx \, dt \right)^{1/2},
$$
where we see 
$$
\dashint_{Q_{\kappa r/2}} |\bar{a}^{ij}(x^1) - a^{ij}(t,x)|^{2p} \, dx \, dt
\le 
N \dashint_{Q_{\kappa r/2}} |\bar{a}^{ij}(x^1) - a^{ij}(t,x)|\, dx \, dt
$$
$$
\le 
N a_{\kappa r/2}^{\#(t,x')}.
$$
From Lemma \ref{lemma04022007} we also see
$$
\left( \dashint_{Q_{\kappa r/2}} |v_{x^i x^j}|^{2p} \, dx \, dt \right)^{1/2}
\le
N(d,p,\delta,\omega)
\left( \dashint_{Q_{\kappa r}} |v_{xx}|^{p} \, dx \, dt \right).
$$
Hence
$$
\dashint_{Q_r} | v_{t}(t,x) - \left( v_{t} \right)_{Q_r} |^p
\, dx \, dt
+
\dashint_{Q_r} | v_{xx'}(t,x) - \left( v_{xx'} \right)_{Q_r} |^p
\, dx \, dt
$$
$$
\le N \left(\kappa^{-p} + \kappa^{d+2} 
\hat{a}^{1/2} \right) \left( |v_{xx}|^p \right)_{Q_{\kappa r}}.
$$
Note that
$$
\left( |v_{xx}|^p \right)_{Q_{\kappa r}}
\le N \left( |u_{xx}|^p \right)_{Q_{\kappa r}}
+ N \left( |w_{xx}|^p \right)_{Q_{\kappa r}}
$$
$$
\le N \left( |u_{xx}|^p \right)_{Q_{\kappa r}}
+ N \left( |f|^p \right)_{Q_{\kappa r}},
$$
where the second inequality is due to \eqref{04022007_02}.
Also note that, using the inequality \eqref{04022007_03},
$$
\dashint_{Q_r}| w_{xx'}(t,x) - \left( w_{xx'} \right)_{Q_r} |^p \, dx \, dt
\le N \left( |w_{xx'}|^p \right)_{Q_r}
\le N \kappa^{d+2} \left( |f|^p \right)_{Q_{\kappa r}},
$$
$$
\dashint_{Q_r}| w_{t}(t,x) - \left( w_{t} \right)_{Q_r} |^p \, dx \, dt
\le N \left( |w_{t}|^p \right)_{Q_r}
\le N \kappa^{d+2} \left( |f|^p \right)_{Q_{\kappa r}}.
$$
Therefore,
$$
\dashint_{Q_r}
| u_{xx'}(t,x) - \left( u_{xx'} \right)_{Q_r} |^p \, dx \, dt
$$
$$
\le
N \dashint_{Q_r}
| v_{xx'}(t,x) - \left( v_{xx'} \right)_{Q_r} |^p \, dx \, dt
+ 
N \dashint_{Q_r}
| w_{xx'}(t,x) - \left( w_{xx'} \right)_{Q_r} |^p \, dx \, dt
$$
$$
\le N \left(\kappa^{-p} + \kappa^{d+2} 
\hat{a}^{1/2} \right)
\left( |u_{xx}|^p + |f|^p \right)_{Q_{\kappa r}}
+ N \kappa^{d+2} \left( |f|^p \right)_{Q_{\kappa r}}.
$$
Similarly, we have
$$
\dashint_{Q_r}
| u_{t}(t,x) - \left( u_{t} \right)_{Q_r} |^p \, dx \, dt
$$
$$
\le N \left(\kappa^{-p} + \kappa^{d+2} 
\hat{a}^{1/2} \right)
\left( |u_{xx}|^p + |f|^p \right)_{Q_{\kappa r}}
+ N \kappa^{d+2} \left( |f|^p \right)_{Q_{\kappa r}}.
$$
The theorem is proved.
\end{proof}

If $g$ is a function defined on $\bR$,
by $(g)_{(\sfa,\sfb)}$ we mean
$$
(g)_{(\sfa, \sfb)} = \dashint_{(\sfa, \sfb)} g(s) \, ds
= (\sfb - \sfa)^{-1} \int_{\sfa}^{\sfb} g(s) \, ds.
$$
The maximal and sharp function of $g$ are defined by
$$
M g (t) = \sup_{t \in (\sfa, \sfb)} \dashint_{(\sfa, \sfb)} |g(s)| \, ds,
$$
$$
g^{\#} (t) = \sup_{t \in (\sfa, \sfb)} \dashint_{(\sfa, \sfb)} | g(s) - (g)_{(\sfa, \sfb)} | \, ds,
$$
where the supremums are taken over all intervals $(\sfa, \sfb)$
containing $t$.

The following corollary follows from Theorem \ref{theorem04082007} above
and the argument in the proof of Corollary 3.2 in \cite{Krylov_2007_mixed_VMO}.

\begin{corollary}							\label{corollary04092007}
Let $p \ge 2$.
In case $p = 2$, we assume that the coefficients $a^{ij}(t,x)$ of $L_0$ are independent of $x' \in \bR^{d-1}$.
Then there exists a constant $N$, depending on $d$, $p$, $\delta$, and 
the function $\omega$, such that,
for any $u \in C_0^{\infty}(\bR^{d+1})$, $\kappa \ge 16$,
and $r \in (0, 1/\kappa]$,
we have
$$
\dashint_{(0,r^2)} \left| \phi(t) - (\phi)_{(0,r^2)} \right|^p \, dt 
+ \dashint_{(0,r^2)} \left| \varphi(t) - (\varphi)_{(0,r^2)} \right|^p \, dt
$$
$$
\le N \kappa^{d+2} (\psi^p)_{(0,(\kappa r)^2)}
+ N \left( \kappa^{-p} + \kappa^{d+2} \hat{a}^{1/2} \right) (\zeta^p)_{(0,(\kappa r)^2)},
$$
where $\hat{a} = a_{\kappa r}^{\#(t,x')}$,
$$
\phi(t) = \| u_{t}(t, \cdot) \|_{L_p(\bR^d)},
\quad
\varphi(t) = \| u_{xx'}(t, \cdot) \|_{L_p(\bR^d)},
$$
$$
\zeta(t) = \| u_{xx}(t, \cdot) \|_{L_p(\bR^d)},
\quad
\psi(t) = \| L_{0} u(t, \cdot) \|_{L_p(\bR^d)}.
$$
\end{corollary}

The following two assertions
are similar to Lemma 3.3 and 3.4 in \cite{Krylov_2007_mixed_VMO}.
However, since our statements are a little bit different from those in \cite{Krylov_2007_mixed_VMO}, we present here proofs.

\begin{lemma}							\label{lemma04102007}
Let $p \ge 2$.
In case $p = 2$, we assume that the coefficients $a^{ij}(t,x)$ of $L_0$ are independent of $x' \in \bR^{d-1}$.
Let $R \in (0,1]$ and $u$ be a function in $C_0^{\infty}(\bR^{d+1})$
such that $u(t,x) = 0$ for $t \notin (0,R^4)$.
Then
$$
\phi^{\#}(t_0) + \varphi^{\#}(t_0)
\le N \kappa^{(d+2)/p} \left(M \psi^p (t_0)\right)^{1/p}
+ N (\kappa R)^{2-2/p} \left(M \phi^p (t_0)\right)^{1/p}
$$
$$
+ N \left( (\kappa R)^{2-2/p} + \kappa^{-1} + \kappa^{(d+2)/p} \left(\omega(R)\right)^{1/2p}   \right) \left(M \zeta^p (t_0)\right)^{1/p}
$$
for all $\kappa \ge 16$ and $t_0 \in \bR$,
where $N = N(d, p, \delta,\omega)$
and the functions $\phi$, $\varphi$, $\zeta$, $\psi$
are defined as in Corollary \ref{corollary04092007}.
\end{lemma}

\begin{proof}
Take a $\kappa$ such that $\kappa \ge 16$.
If $r \le R/\kappa$,
then $\kappa r \le R \le 1$ and $a^{\#(t,x')}_{\kappa r} \le a^{\#(t,x')}_{R}
\le \omega(R)$.
Thus by Corollary \ref{corollary04092007},
$$
\dashint_{(0,r^2)} \left| \phi(t) - (\phi)_{(0,r^2)} \right|^p \, dt 
+ \dashint_{(0,r^2)} \left| \varphi(t) - (\varphi)_{(0,r^2)} \right|^p \, dt
$$
$$
\le N \kappa^{d+2} (\psi^p)_{(0,(\kappa r)^2)}
+ N \left( \kappa^{-p} + \kappa^{d+2} \left(\omega(R)\right)^{1/2} \right) (\zeta^p)_{(0,(\kappa r)^2)}.
$$
An appropriate translation of this inequality
gives us 
$$
\dashint_{(\sfa,\sfb)} \left| \phi(t) - (\phi)_{(\sfa,\sfb)} \right|^p \, dt 
+ \dashint_{(\sfa,\sfb)} \left| \varphi(t) - (\varphi)_{(\sfa,\sfb)} \right|^p \, dt
$$
$$
\le N \kappa^{d+2} (\psi^p)_{(\sfa,\sfc)}
+ N \left( \kappa^{-p} + \kappa^{d+2} \left(\omega(R)\right)^{1/2} \right) (\zeta^p)_{(\sfa,\sfc)}
$$
if $(\sfa, \sfb)$ is an interval such that
$\sfb - \sfa \le R^2/\kappa^2$
and $\sfc = \sfa + \kappa^2 (\sfb - \sfa)$.
Note that, for $t_0 \in (\sfa, \sfb)$,
$$
(\psi^p)_{(\sfa,\sfc)}
\le M \psi^p (t_0),
\quad
(\zeta^p)_{(\sfa,\sfc)}
\le M \zeta^p (t_0).
$$
Thus by using the H\"{o}lder's inequality it follows that
$$
\dashint_{(\sfa,\sfb)} \left| \phi(t) - (\phi)_{(\sfa,\sfb)} \right| \, dt 
+ \dashint_{(\sfa,\sfb)} \left| \varphi(t) - (\varphi)_{(\sfa,\sfb)} \right| 
$$
$$
\le N \kappa^{(d+2)/p} \left( M \psi^p (t_0) \right)^{1/p}
+ N \left( \kappa^{-1} + \kappa^{(d+2)/p} \left(\omega(R)\right)^{1/2p} \right) \left( M \zeta^p (t_0)\right)^{1/p},
$$
where $t_0 \in (\sfa, \sfb)$ and $\sfb - \sfa \le R^2/\kappa^2$.
Now take an interval $(\sfa, \sfb)$ such that $t_0 \in (\sfa, \sfb)$
and $\sfb - \sfa > R^2/\kappa^2$.
Then
$$
\dashint_{(\sfa,\sfb)} \left| \phi(t) - (\phi)_{(\sfa,\sfb)} \right| \, dt 
\le 2 \dashint_{(\sfa,\sfb)} I_{(0,R^4)}(t) \left| \phi(t) \right| \, dt 
$$
$$
\le 2 \left(\dashint_{(\sfa,\sfb)} I_{(0,R^4)}(t) \, dt \right)^{1-1/p}
\left(\dashint_{(\sfa,\sfb)} \left| \phi(t) \right|^p \, dt \right)^{1/p}
$$
$$
\le 2 \left(R^4(\sfb - \sfa)^{-1}\right)^{1-1/p}
\left( M \phi^p (t_0) \right)^{1/p}
\le 2 \left(\kappa R\right)^{2-2/p} \left( M \phi^p (t_0) \right)^{1/p}.
$$
By a similar calculation, we obtain
$$
\dashint_{(\sfa,\sfb)} \left| \varphi(t) - (\varphi)_{(\sfa,\sfb)} \right| \, dt
\le 2 \left(\kappa R\right)^{2-2/p} \left( M \varphi^p (t_0) \right)^{1/p}
$$
$$
\le N \left(\kappa R\right)^{2-2/p} \left( M \zeta^p (t_0) \right)^{1/p}.
$$
Therefore,
for all intervals $(\sfa, \sfb) \ni t_0$,
$$
\dashint_{(\sfa,\sfb)} \left| \phi(t) - (\phi)_{(\sfa,\sfb)} \right| \, dt 
+ \dashint_{(\sfa,\sfb)} \left| \varphi(t) - (\varphi)_{(\sfa,\sfb)} \right| 
$$
$$
\le N \kappa^{(d+2)/p} \left( M \psi^p (t_0) \right)^{1/p}
+ N \left(\kappa R\right)^{2-2/p} \left( M \phi^p (t_0) \right)^{1/p}
$$
$$
+ N \left(\left(\kappa R\right)^{2-2/p} + \kappa^{-1} + \kappa^{(d+2)/p} \left(\omega(R)\right)^{1/2p} \right) \left( M \zeta^p (t_0)\right)^{1/p}.
$$
Taking the supremum of the left-hand side of the above inequality
over all intervals $(\sfa, \sfb) \ni t_0$, we obtain the inequality in the lemma.
The lemma is proved.
\end{proof}

\begin{corollary}							\label{corollary04102007}
Let $q > p \ge 2$.
Assume that, in case $p = 2$,
the coefficients $a^{ij}$ of $L_0$
are independent of $x' \in \bR^{d-1}$.
Then there exists $R = R(d, p, q, \delta, \omega)$
such that,
for any $u \in C_0^{\infty}(\bR^{d+1})$
satisfying $u(t,x) = 0$ for $t \notin (0, R^4)$,
$$
\|u_t\|_{L_{q,p}}
+ \|u_{xx}\|_{L_{q,p}}
\le N \| L_0 u \|_{L_{q,p}},
$$
where $N = N(d, p, q, \delta, \omega)$.
\end{corollary}

\begin{proof}
Let $u \in C_0^{\infty}(\bR^{d+1})$
be a function such that $u(t,x) = 0$ for $t \notin (0, R^4)$,
$R \in (0,1]$,
where $R$ will be specified below.
Using the inequality in Lemma \ref{lemma04102007}
as well as the Hardy-Littlewood theorem
and Fefferman-Stein theorem (note that $q/p > 1$),
we arrive at
$$
\|u_t\|_{L_{q,p}}
+ \|u_{xx'}\|_{L_{q,p}}
\le N \kappa^{(d+2)/p} \| L_0 u \|_{L_{q,p}}
+ N (\kappa R)^{2-2/p} \| u_t \|_{L_{q,p}}
$$
$$
+ N \left( (\kappa R)^{2-2/p} + \kappa^{-1} + \kappa^{(d+2)/p} \left(\omega(R)\right)^{1/2p}   \right) \| u_{xx} \|_{L_{q,p}}
$$
for all $\kappa \ge 16$.
The left-hand side of the above inequality can be replaced by $\|u_t\|_{L_{q,p}}
+ \|u_{xx}\|_{L_{q,p}}$
since
$$
u_{x^1x^1} = \frac{1}{a^{11}} \left(L_0 u  - u_t - 
\sum_{i \ne 1, j \ne 1} a^{ij} u_{x^ix^j}\right).
$$
Now
we choose a large $\kappa$ and then a small $R$ such that
$$
N \left( (\kappa R)^{2-2/p} + \kappa^{-1} + \kappa^{(d+2)/p} \left(\omega(R)\right)^{1/2p}   \right)  < 1/2
$$
$$
N (\kappa R)^{2-2/p} < 1/2.
$$
It then follows that
$$
\|u_t\|_{L_{q,p}}
+ \|u_{xx}\|_{L_{q,p}}
\le 
2 N \kappa^{(d+2)/p} \| L_0 u \|_{L_{q,p}}.
$$
This finishes the proof.
~\end{proof}

Now that we have an $L_{q,p}$-estimate for
functions with compact support with respect to $t \in \bR$,
by repeating word for word the proofs in section 3 in \cite{Krylov_2007_mixed_VMO},
more precisely, proofs
of Lemma 3.4 and Theorem 3.5
in \cite{Krylov_2007_mixed_VMO},
we complete the proof of Theorem \ref{theorem04022007}.

\section{Equations with $a^{ij}$ measurable in $(t,x^1) \in \bR^2$}							\label{section04082007_01}

Throughout this section, we set
$$
\fL_{\lambda} u = u_t + a^{ij}(t,x^1) u_{x^i x^j} - \lambda u,
$$
where $\lambda \ge 0$
and $a^{ij}$ are functions of only $(t,x^1) \in \bR^2$,
$a^{11}(x^1)$ is a function of $x^1 \in \bR$,
satisfying Assumption \ref{assum_01}.

In this section we call $\fq_r$ the 1-spatial dimensional version of $Q_r$,
that is, 
$$
\fq_r(t,x^1) = (t, t+r^2) \times (x^1-r, x^1+r) \subset \bR \times \bR.
$$
Especially, $\fq_r = \fq_r(0,0)$.

As is seen in \cite{Doyoon:parabolic:2006},
one of key steps there is based on Theorem \ref{theorem04022007} in this paper.
Now that we have proved Theorem \ref{theorem04022007}, 
using the results in \cite{Doyoon:parabolic:2006}
as well as in \cite{Doyoon&Krylov:parabolic:2006},
we are able to state the following theorem.

\begin{theorem}							\label{theorem03262007}
Let $p \ge 2$ and $T \in [-\infty, \infty)$.
For any $\lambda>0$ and $f\in L_{p}((T, \infty) \times \bR^{d})$,
there exists a unique solution $u \in W^{1,2}_{p}((T, \infty) \times \bR^{d})$ 
to the equation
$\fL_{\lambda}u = f$. 
Furthermore, there is a constant $N=N(d, p, \delta)$
such that, for any $\lambda \ge 0$ and
$u \in W^{1,2}_{p}((T, \infty) \times \bR^{d})$,
we have
$$
\|u_{t}\|_{L_{p}((T, \infty) \times \bR^{d})}
+\|u_{xx}\|_{L_{p}((T, \infty) \times \bR^{d})}
+\sqrt{\lambda}\|u_{x}\|_{L_{p}((T, \infty) \times \bR^{d})}
$$
$$
+\lambda\|u\|_{L_{p}((T, \infty) \times \bR^{d})}
\le N\|\fL_{\lambda} u\|_{L_{p}((T, \infty) \times \bR^{d})}.
$$
\end{theorem}

More precisely, this theorem follows, in case $p = 2$, from
Theorem 3.2 in \cite{Doyoon&Krylov:parabolic:2006}
and, in case $p > 2$,
from
Corollary 4.2 in \cite{Doyoon:parabolic:2006}
as well as the argument in the proof of Theorem 4.1 in \cite{Krylov_2005}
(see also the discussion following Theorem \ref{theorem04012007}).

Based on the estimate in the above theorem,
we have the following lemmas which are similar to those in section \ref{section04082007}.
However, we do not have $D^{m}_t v$, $m \ge 2$, 
in Lemma \ref{lemma03202007_03}
because $a^{ij}$ are not independent of $t \in \bR$ except $a^{11}$.

\begin{lemma} \label{lemma03202007_01}
Let $p \in [2, \infty)$.
For any $u \in W_{p, \text{loc}}^{1,2}(\bR^{d+1})$,
we have
$$
\|u_t\|_{L_p(Q_r)}
+ \| u_{xx} \|_{L_p(Q_r)}
+ \| u_{x} \|_{L_p(Q_r)}
\le N \left( \| \fL_0 u \|_{L_p(Q_R)} + \| u \|_{L_p(Q_R)} \right),
$$
where $0 < r < R < \infty$ and $N = N(d, p, \delta, r, R)$.
\end{lemma}

\begin{lemma} \label{lemma03202007_03}
Let $p \in [2, \infty)$, $0 < r < R < \infty$, and
$\gamma = (\gamma^1, \cdots, \gamma^d)$ be a multi-index such that $\gamma^1 = 0, 1, 2$. Set $\gamma' = (0, \gamma^2, \cdots, \gamma^d)$.
If $v \in C_{\text{loc}}^{\infty}(\bR^{d+1})$ is a function such that $\fL_0 v = 0$ in $Q_R$, then 
\begin{equation}							\label{03262007_02}
\int_{Q_r} | D^{\gamma'} v_t |^p \, dx \, dt
+ \int_{Q_r} | D^{\gamma} v |^p \, dx \, dt
\le N 
\int_{Q_R} |v|^p \, dx \, dt,
\end{equation}
where $N = N(d, p, \delta, \gamma, r, R)$.
\end{lemma}

Let us recall some function spaces which we need in the following.
We denote by $H_p^{s}(\bR^{d})$, $s \in \bR$, 
as is well-known,
the space of all generalized functions $u$ such that
$(1-\Delta)^{s/2} u \in L_p(\bR^d)$.
For $k = 0,1,2,\cdots$, 
$W_p^{k}(\Omega)$ is the usual Sobolev space and 
$C^{k+\nu}(\Omega)$, $0 < \nu < 1$, is the H\"{o}lder space.
By $C^{k}(\Omega)$ we mean the space of all functions $u$ 
whose derivatives $D^{\alpha} u$, $|\alpha| \le k$, are continuous
and bounded in $\Omega$. 
As usual, we set
$$
\| u \|_{C^{k+\nu}(\Omega)}
= \| u \|_{C^{k}(\Omega)}
+ \sum_{|\alpha| = k} \sup_{\substack{x, y \in \Omega \\ x \ne y}} 
\frac{|D^{\alpha} u(x) - D^{\alpha} u(y)|}{|x-y|^{\nu}},
$$
where
$$
\| u \|_{C^{k}(\Omega)}
= \sum_{|\alpha| \le k} \sup_{x \in \Omega} |D^{\alpha} u(x)|.
$$

The following three lemmas generalize
Lemma 3.4, 3.5, and 3.7 in \cite{Doyoon:parabolic:2006}
to the case $p \ge 2$.

\begin{lemma} \label{lemma03202007_02}
Let $p \ge 2$ and $u \in W_p^{1,2}((0,\infty) \times \bR^d) \cap C^{2}([0,\infty) \times \bR^d)$.
Then
$$
\sup_{0 \le s < \infty} \| u(s,\cdot) \|_{W_{p}^{1}(\bR^d)}
\le 
N(d,p) \| u \|_{W_{p}^{1,2}((0,\infty) \times \bR^d)}.
$$
\end{lemma}

\begin{proof}
Note that 
$$
\int_{\bR^d} |u(s,x)|^p \, dx
= - \int_{\bR^d} \int_s^{\infty} p |u(t,x)|^{p-2} u(t,x) u_t(t,x) \, dt \, dx
$$
and
$$
- p | u(t,x) |^{p-2} u(t,x) u_t(t,x)
\le p |u|^{p-1} |u_t|
\le (p-1) |u|^{p} + |u_t|^p.
$$
Hence
$$
\int_{\bR^d} |u(s,x)|^p \, dx
\le N(p) \| u \|^p_{W_p^{1,2}((0,\infty) \times \bR^d)}
$$
for all $s \in [0, \infty)$.

Similarly,
$$
\int_{\bR^d} |u_{x^i}(s,x)|^p \, dx
= - \int_{\bR^d} \int_s^{\infty} p | u_{x^i}(t,x) |^{p-2} u_{x^i}(t,x) u_{x^i t}(t,x) \, dt \, dx 
$$
$$
= \int_s^{\infty} \int_{\bR^d} (p^2 - p) | u_{x^i}(t,x) |^{p-2} u_{x^i x^i}(t,x) u_t(t,x) \, dx \, dt,
$$
where the last equality is due to integration by parts
(also note that $f'(\xi) = p(p-1) |\xi|^{p-2}$ if $f(\xi) = p |\xi|^{p-2}\xi$, $\xi \in \bR$, $p \ge 2$).
Observe that
$$
(p^2 - p) | u_{x^i} |^{p-2} u_{x^i x^i} u_t
\le (p^2 - p) |u_{x^i x^i}| |u_{x^i}|^{(p-2)/2} |u_{x^i}|^{(p-2)/2} |u_t| 
$$
$$
\le N(p) \left(|u_{x^i x^i}|^2 |u_{x^i}|^{p-2} 
+ |u_{x^i}|^{p-2} |u_t|^2 \right)
$$
$$
\le N(p) \left( |u_{x^i x^i}|^{p} + |u_{x^i}|^{p}
+ |u_{x^i}|^{p} + |u_t|^{p} \right).
$$
Therefore,
$$
\int_{\bR^d} |u_{x^i}(s,x)|^p \, dx
\le N \| u \|^p_{W_p^{1,2}((0,\infty) \times \bR^d)}
$$
for all $s \in [0,\infty)$.
The lemma is proved.
\end{proof}

\begin{lemma} \label{lemma03202007_04}
Let $1 \le q < \infty$, $p \ge 2$,
and $v \in C_{\text{loc}}^{\infty}(\bR^{d+1})$
be a function such that
$\fL_{0} v = 0$ in $Q_{4}$.
Then, for all $x' \in B'_{1}$,
$$
\| v(\cdot, x') \|_{L_{q,p}(\fq_2)}
+ \| v_{x}(\cdot, x') \|_{L_{q,p}(\fq_2)}
+ \| v_{xx'}(\cdot, x') \|_{L_{q,p}(\fq_2)}
\le N \| v \|_{L_p(Q_{R})},
$$
where $3 < R \le 4$, $N=N(d, \delta, p, R)$,
and, as we recall, for example,
$\|v(\cdot,x')\|_{L_{q,p}(\fq_2)}$
is the $L_{q,p}$-norm of $v(t,x^1,x')$ as a function of 
$(t,x^1)$ on $\fq_2 = (0,4) \times (-2,2)$.
\end{lemma}

\begin{proof}
We prove that, for each $x' \in B'_1$,
\begin{equation}							\label{03202007_01}
\|v(\cdot,x')\|_{L_{q,p}(\fq_2)}
+ \|v_{x^1}(\cdot,x')\|_{L_{q,p}(\fq_2)}
\le N \| v \|_{L_p(Q_{\tau})},
\end{equation}
where $3 < \tau < R$ and $N = N(d, \delta, p, \tau)$.
If this turns out to be true, then using this and the fact that $\fL v_{x'} = 0$
in $Q_4$
we obtain the inequality \eqref{03202007_01} with $v_{x'}$ in place of $v$.
Furthermore, using $\fL v_{x'x'} = 0$ in $Q_4$, 
we also obtain the inequality \eqref{03202007_01} with $v_{x'x'}$ in place of $v$. 
Hence the left side of the inequality in the lemma is not greater than a constant times
$$
\| v \|_{L_p(Q_{\tau})} + \| v_{x'} \|_{L_p(Q_{\tau})} + \| v_{x'x'} \|_{L_p(Q_{\tau})}.
$$
This and Lemma \ref{lemma03202007_03} finish the proof.

To prove \eqref{03202007_01}, we introduce 
an infinitely differentiable function $\eta$ defined on $\bR^2$ such that 
$$
\eta(t,x^1) = \left\{
\begin{aligned}
1 & \quad \text{on} \quad [0,4] \times [-2,2]\\
0 & \quad \text{on} \quad \bR^2 \setminus \left[(-r^2,r^2) \times (-r,r)\right]
\end{aligned}
\right.,
$$
where $2 < r < 2\sqrt{2}$.
For each $x' \in B'_1$, view $\eta v$ as a function of $(t,x^1) \in (0,\infty) \times \bR$. 
Then by Lemma \ref{lemma03202007_02}
\begin{equation}							\label{03202007_02}
\sup_{0 \le s < \infty} \| (\eta v)(s,\cdot,x') \|_{W_{p}^{1}(\bR)}
\le N(p) \| (\eta v)(\cdot, x') \|_{W_p^{1,2}((0,\infty) \times \bR)}
\end{equation}
for all $x' \in B'_1$.
Note that the left-hand side of the inequality \eqref{03202007_01}
is less than or equal to a constant times
the left-hand side of the inequality \eqref{03202007_02}.
Moreover, the right-hand side of the inequality \eqref{03202007_02}
is no greater than a constant times 
\begin{equation}							\label{03202007_03}
\| v(\cdot, x') \|_{W_{p}^{1,2}(\fq_r)}.
\end{equation}
for all $x' \in B'_1$.
Now notice that
there exists a constant $N$ and an integer $k$ such that, 
for each $(t,x^1) \in (0,r^2) \times (-r,r)$,
$$
\sup_{x' \in B'_1} |v(t,x^1,x')| \le N \|v(t,x^1,\cdot)\|_{W_{p}^{k}(B'_1)}.
$$
This inequality remains true if we replace $v$ with $v_t$, $v_{x^1}$, or $v_{x^1 x^1}$. 
Hence, for all $x' \in B'_1$, 
the term \eqref{03202007_03} is not greater than a constant times
$$
\left(\int_0^{r^2} \int_{-r}^{r} 
\|v_t(t,x^1,\cdot)\|^p_{W_{p}^{k}(B'_1)}
+ \sum_{m=0}^{2} \|D_{x^1}^m v (t,x^1,\cdot)\|^p_{W_{p}^{k}(B'_1)}
\, dx^1 \, dt \right)^{1/p},
$$
where $D_{x^1}^0 v = v$, $D_{x^1}^1v = v_{x^1}$,
and $D_{x^1}^2 = v_{x^1x^1}$.
We see that the above term is, by Lemma \ref{lemma03202007_03}, less than or equal to a constant times the left-hand side of the inequality \eqref{03202007_01}
(note that $\{ (t,x^1,x'): t \in (0,r^2), x^1 \in (-r,r), |x'|<1\}
\subsetneq Q_{\tau}$).
The lemma is proved.
~\end{proof}

Below we use the following notation.
$$
[f ]_{\mu,\nu; Q_r} := \sup_{\substack{(t,x), (s,y) \in Q_r \\ (t,x) \ne (s,y)}} \frac{|f(t,x) - f(s,y)|}{\,\,\,|t-s|^{\mu} + |x-y|^{\nu}}.
$$

\begin{lemma} \label{lemma03222007}
Let $p \ge 2$, $q > \max\{p, \frac{2p}{p-1}\}$
(note that $\frac{2}{q} < 1 - \frac{1}{p}$),
and $\frac{2}{q} < \beta < 1 - \frac{1}{p}$.
Assume that $v \in C_{\text{loc}}^{\infty}(\bR^{d+1})$
is a function such that
$\fL_{0} v = 0 $ in $Q_{4}$.
Then
$$
[ v_{xx'} ]_{\mu, \nu, Q_1}
\le N(d, p, q, \delta, \beta) \| v \|_{L_p(Q_4)},
$$
where $\mu = \frac{\beta}{2} - \frac{1}{q}$
and $\nu = 1 - \beta - \frac{1}{p}$.
\end{lemma}

\begin{proof}
First we note that
$0 < \mu < 1$ and $0 < \nu < 1$.
We prove
\begin{equation} \label{03212007_01}
[v]_{\mu,\nu,Q_1} + [v_{x^1}]_{\mu,\nu,Q_1} 
\le N \| v \|_{L_p(Q_{\tau})},
\end{equation}
where $3 < \tau < 4$.
If this is done, we can finish the proof using the argument in the proof of Lemma \ref{lemma03202007_04} (i.e., use $\fL v_{x'} = 0$, $\fL v_{x'x'} = 0$, and Lemma \ref{lemma03202007_03}).

To prove the inequality \eqref{03212007_01}
it suffices to prove the following:
for all $s,t \in (0,1)$ and $x' \in B_1'$,
\begin{equation}						\label{h_sup_01}
\| v(t,\cdot,x') - v(s,\cdot,x') \|_{C^{1}(-1,1)} \le N |t-s|^{\mu} 
\| v \|_{L_p(Q_{\tau})},
\end{equation}
\begin{equation}						\label{2006_09_19_01}
\| v(t,\cdot,x') \|_{C^{1+\nu}(-1,1)} + \|v_{x'}(t,\cdot,x')\|_{C^1(-1,1)}
\le N \| v \|_{L_p(Q_{\tau})},
\end{equation}
where $v(t,x^1,x')$ is considered as a function of only $x^1 \in (-1,1)$.
Indeed, observe that,
for $(t,x), (s,y) \in Q_1$,
$$
|v_{x^1}(t,x) - v_{x^1}(s,y)|
\le |v_{x^1}(t,x^1, x') - v_{x^1}(t,y^1,x')|
$$
$$
+ |v_{x^1}(t,y^1,x') - v_{x^1}(t,y^1,y')|
+ |v_{x^1}(t,y^1,y') - v_{x^1}(s,y^1,y')|
$$
$$
\le |x^1 - y^1|^{\nu} \| v(t,\cdot,x') \|_{C^{1+\nu}(-1,1)}
+ \sup_{|z'| < 1} \|v_{x'}(t,\cdot,z')\|_{C^1(-1,1)} |x' - y'|
$$
$$
+ \| v(t,\cdot,y') - v(s,\cdot,y') \|_{C^{1}(-1,1)} 
$$
$$
\le N \left( |x-y|^{\nu} + |t-s|^{\mu} \right) 
\| v \|_{L_p(Q_{\tau})},
$$
where the last inequality is due to \eqref{h_sup_01} and \eqref{2006_09_19_01}.
This proves
$$
[v_{x^1}]_{\mu,\nu,Q_1} \le N \| v \|_{L_p(Q_{\tau})}.
$$
Similarly, from \eqref{h_sup_01} and \eqref{2006_09_19_01} we obtain
$$
[v]_{\mu,\nu,Q_1} \le N \| v \|_{L_p(Q_{\tau})}.
$$

Now we prove the inequalities \eqref{h_sup_01} and \eqref{2006_09_19_01}.
Let $\eta$ be an infinitely differentiable function defined on $\bR^2$ such that 
$$
\eta(t,x^1) = \left\{
\begin{aligned}
1 & \quad \text{on} \quad [0,1] \times [-1,1]\\
0 & \quad \text{on} \quad \bR^2 \setminus (-4,4) \times (-2,2)
\end{aligned}
\right..
$$
Also let
$$
g(t,x^1,x') = - \sum_{i \ne 1 \, \text{or} \, j \ne 1} a^{ij}(t,x^1) v_{x^i x^j}(t,x^1,x'),
$$
so that
$$
v_t + a^{11}(x^1) v_{x^1x^1} = g
$$
Then
$$
(\eta v)_t + a^{11}(x^1) (\eta v)_{x^1 x^1} = 
\eta g + 2 a^{11} \eta_{x^1} v_{x^1} + (\eta_t + a^{11} \eta_{x^1 x^1}) v.
$$
For each $x' \in B'_{1}$, consider $\eta v$ as a function of $(t,x^1) \in (0,\infty) \times \bR$.
Then by Theorem \ref{theorem04022007}
(note that $\eta v = 0$ for $t \ge 4$), 
we have
$$
\| \eta v \|_{W_{q,p}^{1,2}((0,\infty) \times \bR)} 
\le N \|\eta g + 2 a^{11} \eta_{x^1} v_{x^1} + (\eta_t + a^{11} \eta_{x^1 x^1}) v \|_{L_{q,p}((0,\infty) \times \bR)},
$$
where $N = N(\delta, p,q)$.
We see that, for each $x' \in B'_1$, the right hand side of the above inequality is not greater than
a constant times
$$
\| v(\cdot, x') \|_{L_{q,p}(\fq_2)} + \| v_x(\cdot, x') \|_{L_{q,p}(\fq_2)} 
+ \| v_{xx'}(\cdot, x') \|_{L_{q,p}(\fq_2)},
$$
which is, by Lemma \ref{lemma03202007_04}, less than or equal to a constant times
$$
\| v \|_{L_p(Q_{r})},
$$
where $3 < r < \tau$. 
Hence
\begin{equation}\label{sup_esti_0001}
\| (\eta v)(\cdot, x') \|_{W_{q,p}^{1,2}((0,\infty) \times \bR)} 
\le N \| v \|_{L_p(Q_{r})}
\end{equation}
for all $x' \in B'_1$.
Again we view $\eta v$ as a function of $(t,x^1) \in (0,\infty) \times \bR$.
Then by Theorem 7.3 in \cite{MR1837532}
\begin{equation}\label{sup_esti_0002}
\| (\eta v)(t,\cdot,x') - (\eta v)(s,\cdot,x')\|_{H_p^{2-\beta}(\bR)}
\le N |t-s|^{\mu}
\| (\eta v)(\cdot, x') \|_{W_{q,p}^{1,2}((0,\infty) \times \bR)}
\end{equation}
for each $x' \in B'_1$,
where $N$ is independent of $s$, $t$, and $\eta v$.
Using an embedding theorem, we have
$$
\|(\eta v)(t,\cdot,x') - (\eta v)(s,\cdot,x')\|_{C^{1+\nu}(\bR)}
$$
$$
\le N \| (\eta v)(t,\cdot,x') - (\eta v)(s,\cdot,x')\|_{H_p^{2-\beta}(\bR)},
$$
where, as noted earlier, $\nu = 1 - \beta - 1/p$.
From this, \eqref{sup_esti_0001}, and \eqref{sup_esti_0002},
we finally have
$$
\|(\eta v)(t,\cdot,x') - (\eta v)(s,\cdot,x')\|_{C^{1+\nu}(\bR)}
\le N |t-s|^{\mu} \| v \|_{L_p(Q_{r})}
$$ 
for all $x' \in B'_1$. This proves \eqref{h_sup_01}.
Now 
by setting $s = 4$ in the above inequality, 
we obtain
\begin{equation} \label{h_sup_02}
\| v(t,\cdot,x') \|_{C^{1+\nu}(-1,1)} 
\le N \| v \|_{L_p(Q_{r})}.
\end{equation}
Then using the above inequality and the fact that $\fL v_{x'} = 0$ in $Q_4$,
we have
$$
\| v_{x'}(t,\cdot,x') \|_{C^{1+\nu}(-1,1)} \le N 
\| v_{x'} \|_{L_p(Q_{r})}.
$$
This and \eqref{h_sup_02} along with Lemma \ref{lemma03202007_03}
prove \eqref{2006_09_19_01} (recall that $3 < r < \tau$).
The lemma is proved.
~\end{proof}

Lemma \ref{lemma04012007_03} and \ref{lemma04012007_04} in section \ref{section04082007}
are repeated below, 
but since the operator $\fL_{\lambda}$ is being dealt with,
the lemmas have to be modified as follows.

\begin{lemma} \label{lemma03262007}
Let $p \ge 2$, $q > \max\{p, \frac{2p}{p-1}\}$,
and $\frac{2}{q} < \beta < 1 - \frac{1}{p}$.
For every $v \in C_{\text{loc}}^{\infty}(\bR^{d+1})$
such that $\fL_{\lambda} v = 0$ in $Q_{4}$,
we have
$$
[ v_{xx'} ]_{\mu, \nu, Q_1}
\le N 
\left( \|v_{xx}\|_{L_p(Q_4)} + \|v_t\|_{L_p(Q_4)} + \sqrt{\lambda} \|v_x\|_{L_p(Q_4)} \right),
$$
where $\mu = \frac{\beta}{2} - \frac{1}{q}$,
$\nu = 1 - \beta - \frac{1}{p}$,
and $N = N(d, p, q, \delta, \beta)$.
\end{lemma}

\begin{proof}
We follow the steps in the proof of Lemma \ref{lemma04012007_03},
but the sup-norms of the derivatives of $v$ on $Q_1$
have to be replaced by $[ v_{xx'} ]_{\mu, \nu, Q_1}$.
\end{proof}

\begin{lemma} \label{lemma03272007}
Let $p \ge 2$, $\lambda \ge 0$, $\kappa \ge 4$, and $r \in (0,\infty)$.
Let $v \in C_{\text{loc}}^{\infty}(\bR^{d+1})$ be 
such that
$\fL_{\lambda}v = 0$ in $Q_{\kappa r}$.
Then there is a constant $N$, depending only on $d$, $p$, and $\delta$,
such that
\begin{equation}							\label{03232007_01}
\dashint_{Q_r} | v_{xx'}(t,x) - \left( v_{xx'} \right)_{Q_r} |^p
\, dx \, dt
\le N \kappa^{-\nu p} \left( |v_{xx}|^p + |v_{t}|^p + \lambda^{p/2} |v_{x}|^p \right)_{Q_{\kappa r}},
\end{equation}
where $\nu = 1/2 - 3/(4p)$.
\end{lemma}

\begin{proof}
We use Lemma \ref{lemma03262007}
with $q = 4p$ and $\beta = 1/2 - 1/(4p)$.
Thus $2\mu = \nu = 1/2 - 3/(4p)$.
Note that
$$
\dashint_{Q_1} | v_{xx'}(t,x) - \left( v_{xx'} \right)_{Q_1} |^p
\, dx \, dt
\le N [v_{xx'}]^p_{\mu,\nu,Q_1}
$$
and
$$
[ \check{v}_{xx'} ]_{\mu, \nu, Q_1}
= \left( \frac{\kappa}{4} \right)^{\nu + 2}
[ v_{xx'} ]_{\mu, \nu, Q_{\kappa/4}}
\quad
\text{if}
\quad
\check{v}(t,x) = v \left(\left(\frac{\kappa}{4}\right)^2 \! t, \, \frac{\kappa}{4} x \right).
$$
Using these as well as the argument in the proof of Lemma \ref{lemma04012007_04},
one can complete the proof.
\end{proof}

Now we arrive at the following theorem, the proof of which
is almost identical to that of Theorem \ref{theorem04022007_01}.

\begin{theorem} \label{theorem03282007}
Let $p \ge 2$. 
Then there is a constant $N$, depending only on $d$, $p$, and $\delta$,
such that, for any $u \in W_p^{1,2}(\bR^{d+1})$,
$r \in (0, \infty)$, and $\kappa \ge 8$,
$$
\dashint_{Q_r} | u_{xx'}(t,x) - \left( u_{xx'} \right)_{Q_r} |^p
\, dx \, dt
\le N \kappa^{d+2} \left( |\fL_0 u|^p \right)_{Q_{\kappa r}}
+ N \kappa^{-\nu p} \left( |u_{xx}|^p \right)_{Q_{\kappa r}},
$$
where $\nu = 1/2 - 3/(4p)$.
\end{theorem}

\section{Proof of Theorem \ref{theorem03192007}} \label{section04092007}

In this section, as in Theorem \ref{theorem03192007},
the coefficients of 
$$
Lu = u_t + a^{ij}(t,x) u_{x^i x^j} + b^{i} u_{x^i} + c u
$$ 
satisfy Assumption \ref{assum_01} and \ref{assum_02}.
Especially, the coefficients $a^{ij}(t,x)$ are independent of $x' \in \bR^{d-1}$ if $p = 2$.
Set 
$$
L_0 u = u_t + a^{ij}(t,x) u_{x^i x^j}.
$$

As noted earlier, due to Theorem \ref{theorem04022007} in this paper,
the results in \cite{Doyoon:parabolic:2006} are now available.
This implies that, by the same reasoning as in the proof of Lemma \ref{lemma04022007},
the inequalities in Lemma \ref{lemma04022007}
are possible with $L_0$ defined above.
Then using the results in section \ref{section04082007_01} and repeating the proof of Theorem \ref{theorem04082007} (with necessary changes), we obtain

\begin{theorem}
Let $p \ge 2$.
In case $p = 2$, we assume that the coefficients $a^{ij}(t,x)$ of $L_0$ are independent of $x' \in \bR^{d-1}$.
Then there exists a constant $N$, depending on $d$, $p$, $\delta$, and 
the function $\omega$, such that,
for any $u \in C_0^{\infty}(\bR^{d+1})$, $\kappa \ge 16$,
and $r \in (0, 1/\kappa]$,
we have
$$
\dashint_{Q_r}
| u_{xx'}(t,x) - \left( u_{xx'} \right)_{Q_r} |^p \, dx \, dt
$$
$$
\le N \kappa^{d+2} \left( |L_0 u|^p \right)_{Q_{\kappa r}}
+ N \left( \kappa^{-\nu p} + \kappa^{d+2} (a_{\kappa r}^{\#})^{1/2} \right) \left( |u_{xx}|^p \right)_{Q_{\kappa r}},
$$
where $\nu = 1/2 - 3/(4p)$
\end{theorem}

In the following we state corollaries and a lemma corresponding
to Corollary \ref{corollary04092007}, Lemma \ref{lemma04102007}, and Corollary \ref{corollary04102007}.
Since we are dealing with $a^{ij}$ different from those in section \ref{section04102007},
we have different statements.

\begin{corollary}							\label{corollary04102007_01}
Let $p \ge 2$.
In case $p = 2$, we assume that the coefficients $a^{ij}(t,x)$ of $L_0$ are independent of $x' \in \bR^{d-1}$.
Then there exists a constant $N$, depending on $d$, $p$, $\delta$, and 
the function $\omega$, such that,
for any $u \in C_0^{\infty}(\bR^{d+1})$, $\kappa \ge 16$,
and $r \in (0, 1/\kappa]$,
we have
$$
\dashint_{(0,r^2)} \left| \varphi(t) - (\varphi)_{(0,r^2)} \right|^p \, dt
$$
$$
\le N \kappa^{d+2} (\psi^p)_{(0,(\kappa r)^2)}
+ N \left( \kappa^{-\nu p} + \kappa^{d+2} (a_{\kappa r}^{\#})^{1/2} \right) (\zeta^p)_{(0,(\kappa r)^2)},
$$
where $\nu = 1/2 - 3/(4p)$,
$$
\varphi(t) = \| u_{xx'}(t, \cdot) \|_{L_p(\bR^d)},
$$
$$
\zeta(t) = \| u_{xx}(t, \cdot) \|_{L_p(\bR^d)},
\quad
\psi(t) = \| L_{0} u(t, \cdot) \|_{L_p(\bR^d)}.
$$
\end{corollary}

\begin{lemma}							\label{lemma04102007_01}
Let $p \ge 2$.
In case $p = 2$, we assume that the coefficients $a^{ij}(t,x)$ of $L_0$ are independent of $x' \in \bR^{d-1}$.
Let $R \in (0,1]$ and $u$ be a function in $C_0^{\infty}(\bR^{d+1})$
such that $u(t,x) = 0$ for $t \notin (0,R^4)$.
Then
$$
\varphi^{\#}(t_0)
\le N \kappa^{(d+2)/p} \left(M \psi^p (t_0)\right)^{1/p}
$$
$$
+ N \left( (\kappa R)^{2-2/p} + \kappa^{-\nu} + \kappa^{(d+2)/p} \left(\omega(R)\right)^{1/2p}   \right) \left(M \zeta^p (t_0)\right)^{1/p}
$$
for all $\kappa \ge 16$ and $t_0 \in \bR$,
where $\nu = 1/2 - 3/(4p)$, $N = N(d, p, \delta,\omega)$,
and the functions $\varphi$, $\zeta$, $\psi$
are defined as in Corollary \ref{corollary04102007_01}.
\end{lemma}

The proof of the next corollary clearly shows
the necessity of the result
for the case with $a^{ij}(t,x)$ measurable in $x^1 \in \bR$
and VMO in $(t,x') \in \bR^d$ (Theorem \ref{theorem04022007},
specifically, Corollary \ref{corollary04102007}).

\begin{corollary}
Let $q > p \ge 2$.
Assume that, in case $p = 2$,
the coefficients $a^{ij}$ of $L_0$
are independent of $x' \in \bR^{d-1}$.
Then there exists $R = R(d, p, q, \delta, \omega)$
such that,
for any $u \in C_0^{\infty}(\bR^{d+1})$
satisfying $u(t,x) = 0$ for $t \notin (0, R^4)$,
$$
\|u_t\|_{L_{q,p}}
+ \|u_{xx}\|_{L_{q,p}}
\le N \| L_0 u \|_{L_{q,p}},
$$
where $N = N(d, p, q, \delta, \omega)$.
\end{corollary}

\begin{proof}
As is seen in Corollary \ref{corollary04102007},
from Lemma \ref{lemma04102007_01},
we obtain
$$
\|u_{xx'}\|_{L_{q,p}}
\le N \kappa^{(d+2)/p} \| L_0 u \|_{L_{q,p}}
$$
$$
+ N \left( (\kappa R)^{2-2/p} + \kappa^{-\nu} + \kappa^{(d+2)/p} \left(\omega(R)\right)^{1/2p}   \right) \| u_{xx} \|_{L_{q,p}}
$$
for all $\kappa \ge 16$ and $R \in (0,1]$,
where $\nu = 1/2 - 3/(4p)$.
To obtain an estimate for $u_{x^1x^1}$, we set
$$
g = L_0 u + \Delta_{d-1} u - \sum_{i \ne 1, j \ne 1} a^{ij}u_{x^ix^j},
$$
where $\Delta_{d-1}u = \sum_{i=2}^d u_{x^ix^i}$.
Then
$$
L_1 u := u_t + a^{11} u_{x^1x^1} + \Delta_{d-1} u = g
$$
and the operator $L_1$ satisfies the assumptions in Corollary \ref{corollary04102007}.
Thus there exist $R_1 = R_1(d, p, q, \delta, \omega)$ 
and $N = N(d, p, q, \delta, \omega)$
such that
$$
\|u_{x^1x^1}\|_{L_{q,p}}
\le N \|g\|_{L_{q,p}}
\le N \left( \|L_0 u\|_{L_{q,p}} + \|u_{xx'}\|_{L_{q,p}}\right)
$$
for all $u \in C_0^{\infty}(\bR^{d+1})$ such that
$u(t,x) = 0$ for $t \notin (0,R_1^4)\times\bR^d$.
From this together with the estimate for $\|u_{xx'}\|_{L_{q,p}}$
above,
we have
$$
\|u_{xx}\|_{L_{q,p}}
\le N \kappa^{(d+2)/p} \| L_0 u \|_{L_{q,p}}
$$
$$
+ N \left( (\kappa R)^{2-2/p} + \kappa^{-\nu} + \kappa^{(d+2)/p} \left(\omega(R)\right)^{1/2p}   \right) \| u_{xx} \|_{L_{q,p}}
$$
if $u(t,x) = 0$ for $t \notin (0, \min\{R^4, R_1^4\})$.
Now
we choose a large $\kappa$ and then a small $R$ (smaller than $R_1$)
such that
$$
N \left( (\kappa R)^{2-2/p} + \kappa^{-\nu} + \kappa^{(d+2)/p} \left(\omega(R)\right)^{1/2p}   \right)  < 1/2
$$
(note that $\nu > 0$).
Then we have
$$
\|u_{xx}\|_{L_{q,p}}
\le 
2 N \kappa^{(d+2)/p} \| L_0 u \|_{L_{q,p}}.
$$
Finally, notice that
$$
\|u_t\|_{L_{q,p}}
= \| L_0 u - a^{ij}u_{x^ix^j} \|_{L_{q,p}}
\le \|L_0 u\|_{L_{q,p}} + N \|u_{xx}\|_{L_{q,p}}.
$$
The corollary is now proved.
~\end{proof}

As in section \ref{section04102007},
using the $L_{q,p}$-estimate proved above for
functions with compact support with respect to $t \in \bR$
and following the proofs in section 3 in \cite{Krylov_2007_mixed_VMO},
we complete the proof of Theorem \ref{theorem03192007}.

\bibliographystyle{plain}

\def\cprime{$'$}\def\cprime{$'$} \def\cprime{$'$} \def\cprime{$'$}
  \def\cprime{$'$} \def\cprime{$'$}

\end{document}